\documentclass[a4paper,twoside,11pt]{article}
\usepackage{a4wide}
\usepackage[english]{babel} 
\usepackage{amsmath,amssymb,amsthm,amsfonts}
\usepackage{ifpdf}
\usepackage{bbm}
\usepackage{graphicx}
\usepackage{hyperref} % always load before showlabels
\usepackage{fancyhdr}
\usepackage{enumitem}

\theoremstyle{plain}
\newtheorem{theorem}{Theorem}[section]
\newtheorem{lemma}[theorem]{Lemma}
\newtheorem{corollary}[theorem]{Corollary}
\newtheorem{proposition}[theorem]{Proposition}

\theoremstyle{definition}

\theoremstyle{remark}
\newtheorem{remark}[theorem]{Remark}
\newtheorem*{remark*}{Remark}

\numberwithin{equation}{section}

\def\N{\ensuremath{\mathbb{N}}}
\def\Q{\ensuremath{\mathbb{Q}}}
\def\R{\ensuremath{\mathbb{R}}}
\def\Z{\ensuremath{\mathbb{Z}}}

\def\ep{\varepsilon}

\def\E{\ensuremath{\mathbb{E}}}
\def\P{\ensuremath{\mathbb{P}}}
\def\F{\ensuremath{\mathcal{F}}}

\def\Ind{\ensuremath{\mathbbm{1}}}

\def\to{\rightarrow}

\def\tas{\ensuremath{\text{ as }}}

\newcommand{\dd}{\mathrm{d}}

\usepackage{color}

\newcommand{\executeiffilenewer}[3]{%
\ifnum\pdfstrcmp{\pdffilemoddate{#1}}%
{\pdffilemoddate{#2}}>0%
{\immediate\write18{#3}}\fi%
}

\usepackage[normalem]{ulem}

\begin{document}

\author{\begin{tabular}{cc}
\parbox{.4\textwidth}{\centering
\textsc{Jean B\'erard} \thanks{ e-mail:  \rm \texttt{jean DOT berard AT univ-lyon1 DOT fr} }  \textsuperscript{,}\footnotemark[3] \\
\emph{Universit\'e de Strasbourg\\ 
	67400 Strasbourg, France}}
	&
\parbox{.4\textwidth}{\centering 
\textsc{Pascal Maillard} \thanks{e-mail:  \rm \texttt{pascal DOT maillard AT weizmann DOT ac DOT il} }  \textsuperscript{,}\thanks{Partially supported by ANR project MEMEMO2. PM is partially supported by a grant from the Israel Science Foundation.}\\
\emph{Weizmann Institute of Science\\
76100 Rehovot, Israel}}
\end{tabular}
\\
\\
  }

\title{The limiting process of $N$-particle branching random walk with polynomial tails}

% definitions of this document
\def\supp{\operatorname{supp}}
\def\Pc{\mathcal{P}}
\def\Rc{\mathcal{R}}
\def\X{\mathcal{X}}
\def\Y{\mathcal{Y}}
\def\logtwo{\log_2}
\def\lN{\lceil \logtwo N \rceil}

\maketitle

\begin{abstract}
We consider a system of $N$ particles on the real line  that evolves through iteration of the following steps: 1) every particle splits into two, 2) each particle jumps according to a prescribed displacement distribution supported on the positive reals and 3) only the $N$ right-most particles are retained, the others being removed from the system. This system has been introduced in the physics literature as an example of a microscopic stochastic model describing the propagation of a front. Its behavior for large $N$ is now well understood -- both from a physical and mathematical viewpoint -- in the case where the displacement distribution admits exponential moments. Here, we consider the case of displacements with regularly varying tails, where the relevant space and time scales are markedly different.
We characterize the behavior of the system for two distinct asymptotic regimes. First, we prove convergence in law of the rescaled positions of the particles on a time scale of order $\log N$ and give a construction of the limit based on the records of a space-time Poisson point process. Second, we determine the appropriate scaling when we let first the time horizon, then $N$ go to infinity.

\bigskip


\bigskip


\end{abstract}

\section{Introduction}
\subsection{Definitions}
\label{sec:definitions}

\paragraph{The $N$-BRW.} 
Let $X$ be a random variable taking values in $\R_+$ and define $h(x)$ by
 \begin{equation}\label{d:def-de-h}\P(X>x) = 1/h(x), \ x \geq 0.\end{equation} 
We assume throughout the paper that the function $h(x)$ is regularly varying at $+\infty$ with index $\alpha > 0$  (see e.g. \cite{Bingham1987} for a definition of regular variation). For every integer $N\ge 1$, we define the following $N$-particle system. At the beginning, all $N$ particles are located at the origin. At each time step, each of the $N$ particles branches into two, and all $2N$ particles then perform independent jumps according to the law of $X$. After this, only the $N$ particles at the maximal positions are retained. We call this system the $N$-branching random walk or $N$-BRW.

More formally, we define a sequence $(\X(n))_{n \geq 0}$ of $N$-tuples of real numbers that represent the successive populations of $N$ particles, with 
$$\X(n) =  \{ \X_1(n) \leq \cdots \leq \X_N(n) \}.$$
Let $(X_{n,i})_{n\ge 0,i\in[2N]}$ denote i.i.d. random variables distributed as $X$. Initially, one sets $\X_i(0) = 0$ for all $i\in[N] :=\{1,\ldots,N\}$. Then, for each integer $n \geq 0$, one inductively defines 
$\X(n+1) = \{ \X_1(n+1) \leq \cdots \leq \X_N(n+1) \}$ to be the $N$ largest numbers from the collection $(\X_i(n)+X_{n,2i+j})_{i\in[N],j\in\{0,1\}}$, sorted in ascending order. 
%Then $\X_i(n)$ denotes the position of the particle at the $i$-th lowest position at time $n$.
%We further define $\X_i(t) = \X_i(\lfloor t\rfloor)$ for every $i\in [N]$ and $t\ge 0$ and set $\X(t) = \{\X_i(t);1\le i\le N\}$.
%We then define the rescaled processes $\Y_i(t) = c_N^{-1}\X_i(t\logtwo N)$, $i\in[N]$.

\paragraph{The stairs process.} We now define the \emph{stairs process}, a real-valued stochastic process which will be shown to approximate the $N$-BRW when $N$ is large. We first define a \emph{stairs measure} to be a non-zero measure $\mu$ on $(0,\infty)$ such that $\mu([a, +\infty))<+\infty$ for every $a > 0$ (in particular, $\mu$ is $\sigma$-finite).
The \(\mu\)-stairs process then is the real-valued stochastic process $\Rc(t)$ defined as follows: Given a Poisson point process on $\{(t,x):t,x>0\}$ with intensity $\dd t\otimes \mu$, define a process \((\xi_{t})_{t\geq0} \) by \(\xi _{t}=x\) if \((t,x)\) is an atom of the process and \(\xi_{t} = 0\) otherwise. Now define $\Rc(t)$ inductively as follows: For $t\le 0$, $\Rc(t) = 0$. For integer $n=0,1,2,\ldots$, knowing the value of $\Rc(t)$ for $t\le n$, define $\Rc(t)$ for $t\in (n,n+1]$ by 
\begin{equation}
\label{eq:def_stairs}
\Rc(t) = \max_{s\in[0,1]}(\Rc(t-s-1) + \xi_{t-s}).
\end{equation}

\begin{figure}[ht]
\begin{center}
 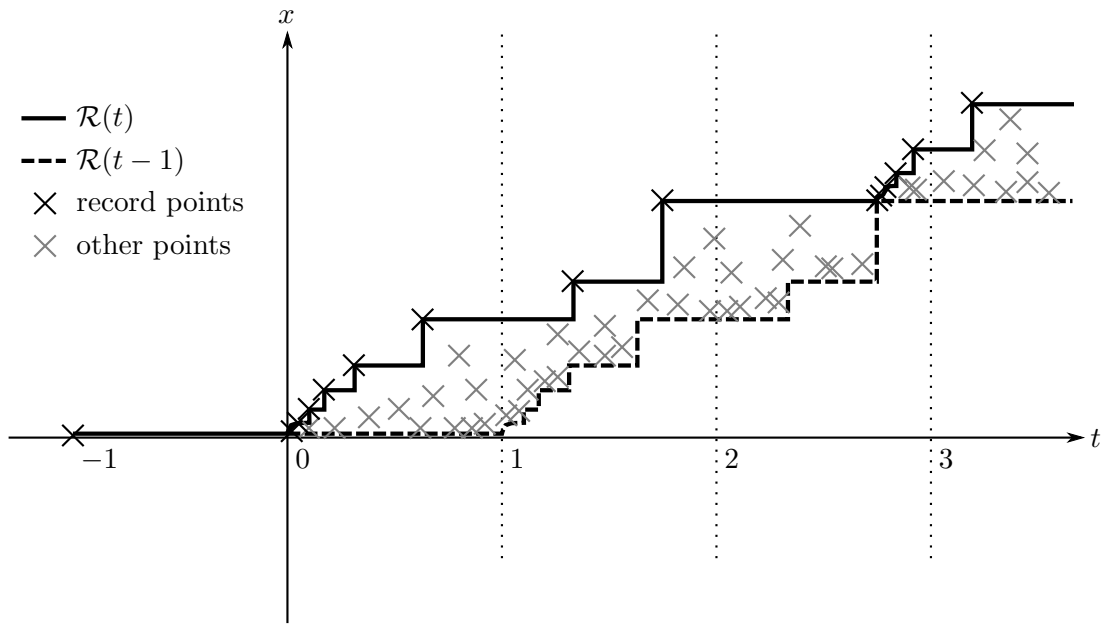
\end{center}
\caption{Graphical representation of the stairs process.}
\label{fig:1}
\end{figure}

This definition is equivalent to the following construction: Suppose the process $\Rc(t)$ is defined for \(t\le n\in\N.\)   Now generate points in the interval \((n,n+1]\) according to the above Poisson process and translate every atom $(t,x)$ by \(\Rc(t-1)\) in the $x$-direction (note that the graph of \(\Rc(t-1)\) is the graph of \(\Rc(t)\) shifted by 1 to the right). Then define \(\Rc(t)\) for \(t\in(n,n+1]\) to be the record process of these points. See Figure~\ref{fig:1} for a graphical representation.

One easily verifies that \(\Rc(t)\) is a non-decreasing c\`adl\`ag process and that the following representation holds for $t\ge0$:
\begin{align}
\label{eq:Rcrep}
\Rc(t) = \max\left\{\sum_{i=1}^k \xi_{t_i}: k\in\N,\ 0\le t_1<\cdots<t_k\le t,\ |t_i-t_{i-1}|\ge 1\ \forall i\right\} 
\end{align}

\begin{remark*}
The process \(\max_{s\in[0,t]}\xi_s\) is known in the literature as a \emph{Poisson paced record process} \cite{Bunge2001}.
The stairs process can be interpreted as a self-interacting version of it. Its long-term behaviour is quite different: While a Poisson paced record process usually grows logarithmically in \(t,\) the stairs process grows like a random walk, due to the existence of regeneration times (see Section~\ref{sec:stairs}).
\end{remark*}

\subsection{Statements of the results}

Define $\mu_\alpha$ to be the measure on $(0,\infty)$ defined by $\mu_\alpha([x,\infty)) = x^{-\alpha}$ and let $\Rc^\alpha(t)$ denote a realisation of the $\mu_\alpha$-stairs process. Define a sequence $(c_N)_{N\ge1}$ by $c_N = h^{-1}(2N\logtwo N)$, where $h^{-1}$ is the generalized inverse of $h$. Note that by the regular variation of $h$, we have $h(c_N)\sim 2N\logtwo N$ as $N\to\infty$. 

Our first theorem gives convergence in law of the maximum and the minimum of the $N$-BRW to a certain stairs process, after rescaling of space and time. For a definition of Skorokhod's $J_1$ and $SM_1$ topologies appearing in the statement of the theorem, see \cite[Chapter 12]{Whitt2002}. Note that the former topology is also commonly called \emph{the} Skorokhod topology.

\begin{theorem}
\label{th:convergence}
We have the following convergences in law, as $N\to\infty$:
\begin{align*}
(c_N^{-1}\X_N(\lfloor t\log_2 N\rfloor))_{t\ge 0}  &\Longrightarrow (\Rc^\alpha(t))_{t\ge0},&&\quad\text{in the $J_1$-topology}\\
(c_N^{-1}\X_N(\lfloor t\log_2 N\rfloor,c_N^{-1}\X_1(\lfloor t\log_2 N\rfloor))_{t\ge 0}  &\Longrightarrow (\Rc^\alpha(t),\Rc^\alpha(t-1))_{t\ge0},&&\quad\text{in the $SM_1$-topology.}
\end{align*}
\end{theorem}

\begin{remark*}
One cannot expect convergence of $c_N^{-1}\X_1(\lfloor t\log_2 N\rfloor)_{t\ge 0}$ in the $J_1$-topology, since for some values of $N$, it may have two consecutive macroscopic jumps.
\end{remark*}

For a random variable $Y$, denote by $\mathcal L(Y)$ the law of $Y$. Denote by $d(\cdot,\cdot)$ the Prokhorov metric on the space of probability measures on $\R$. For two positive sequences $a_N$ and $b_N$, write $a_N\sim b_N$ if $a_N/b_N\to1$ as $N\to\infty$.

The next theorem is our main theorem. It studies the limiting behaviour of the $N$-BRW, when we let first the time horizon, then $N$ go to infinity.

\begin{theorem}
\label{th:main} We distinguish between the following cases:
\begin{itemize}
\item $\alpha > 1$: The limit $v_N = \lim_{n\to\infty} \X_N(n)/n = \lim_{n\to\infty} \X_1(n)/n$ exists almost surely and in \(L^{1}\) and satisfies $v_N \sim \rho_\alpha c_{N}/\logtwo N.$ Here, the limit $\rho_\alpha = \lim_{t\to\infty} \Rc^\alpha(t)/t$ exists almost surely and in $L^1$ and is a positive and finite constant.
\item $\alpha = 1$, $\E[X]<\infty$: The limit $v_N = \lim_{n\to\infty} \X_N(n)/n = \lim_{n\to\infty} \X_1(n)/n$ exists almost surely and in \(L^{1}\) and satisfies $v_N \sim (c_{N}/\logtwo N)\int_1^\infty h(c_N)/h(c_Nx)\,\dd x.$
\item $\alpha = 1$, $\E[X] = \infty$: Set $b^N_n = \int_1^{h^{-1}(n)}h(c_N)/h(c_N x)\,\dd x$. Then, for $i\in\{1,N\}$,
\[
\lim_{N\to\infty} \limsup_{n\to\infty} d\left(\mathcal L\left(\frac{\log_2N}{c_N}\frac{\X_i(n)}{nb^N_n}\right), \delta_1\right) = 0.
\]
% Then for every $\ep>0$ there exists $N_0$, such that for $N\ge N_0$,
%\[
%\limsup_{n\to\infty} \P\left(1-\ep < \frac{\log_2N}{c_N}\frac{\X_1(n)}{nb^N_n} \le \frac{\log_2N}{c_N}\frac{\X_N(n)}{nb^N_n} < 1+\ep\right) > 1-\ep.
%\]
\item $0 < \alpha < 1$:  Let  $W_\alpha$ be a random variable with Laplace transform 
\begin{equation}
\label{eq:Walpha}
\E[e^{-\lambda W_\alpha}] = \exp(-\alpha\int_0^\infty (1-e^{-\lambda x})x^{-\alpha-1}\,\dd x).
\end{equation}
Then, for $i\in\{1,N\}$,
\[
\lim_{N\to\infty} \limsup_{n\to\infty} d\left(\mathcal L\left((2N)^{-\alpha}\frac{\X_i(n)}{h^{-1}(n)}\right), \mathcal L(W_\alpha)\right) = 0.
\]

\end{itemize}
\end{theorem}

\subsection{Heuristics and proof strategy}

The global heuristic picture of the $N$-BRW is the following: At a typical time and viewed on the space scale $c_N$, the particles are divided into one big ``tribe'' located near the left-most particle and containing all but $o(N)$ particles, the remaining particles to the right being split into a $O(1)$ number of smaller tribes. At each time step, the number of particles in every small tribe is multiplied by two, which eventually leads to extinction of the big tribe and another tribe taking over. Furthermore, new tribes are formed by particles performing (rightward) jumps of magnitude $c_N$ out of the big tribe. The value of $c_N$ has been chosen so that these events occur  on the time scale $\log_2 N$, which is precisely the time it takes for a new tribe to grow to size $N$. 

As $N\to\infty$, the jumps leading to new tribes are described by a space-time Poisson point process shifted in space by the position of the big tribe. This leads to the definition of the stairs process from Section~\ref{sec:definitions} and illustrated in Figure~\ref{fig:1}. Note that the record points of that process exactly correspond to the creation of tribes which will eventually take over the population, the other points representing tribes which get extinct before reaching a size of order $N$.

In order to render this picture rigorous, we will couple the $N$-BRW with a discretized version of the stairs process, see \eqref{eq:dsp}, in such a way that the error is bounded in $L^p$ for $p<2\alpha$. Ideally, we would have liked to use a single coupling between both processes to derive Theorems~\ref{th:convergence} and~\ref{th:main} and say something about the empirical measure of the $N$-BRW. However, it turned out to be more convenient to use separate couplings for upper and lower bounds for the positions of the right- and left-most particles. The trickier part here consisted in the upper bound, the key to which is Proposition~\ref{prop:theta}. An important ingredient consists of large deviation estimates for sums of iid variables with regularly varying tails.

Theorem~\ref{th:main} is then derived from the above coupling and Theorem~\ref{th:stairs} below. In order to prove the latter, we define a  regeneration structure for the stairs process, which permits to use classical results on random walks with regularly varying tails.

Note that analogues of the above theorems should remain valid if one allows $X$ to take on negative values or if one considers more general reproduction laws. Since this increases the technical difficulties without leading to new phenomena, we chose to keep to our setting for simplicity.

\subsection{Discussion}

For light-tailed displacement distributions, i.e. satisfying an exponential moment assumption, the $N$-BRW model has been studied in the physics literature as a microscopic stochastic model describing the propagation of a front, along with several variants \cite{Brunet1999,Brunet2001,Brunet2006,Brunet2006a}. In the limit as $N\to\infty$, the dynamics of the model is described by an analog of the Fisher--Kolmogorov--Petrovskii--Piskounov (FKPP) equation, which is a prototypical example of a reaction-diffusion equation admitting travelling wave solutions \cite{Fisher1937,Kolmogorov1937} (see \cite{Durrett2009} for a rigorous result of this type). One is then interested in understanding how the behaviour of the model for large but finite $N$ reflects that of the limiting equation. The main results are the following:
\begin{itemize}
\item The cloud of particles propagates with a finite asymptotic velocity $v_N$. As $N\to\infty$, $v_N$ converges to the velocity $v$ of a travelling wave solution to the corresponding FKPP-type equation, which is also the speed of the right-most particle in the BRW without selection. Moreover, this convergence takes place at the unusually slow rate of $(\log N)^{-2}$ \cite{Brunet1997,Berard2010}.
%\item The diameter of the cloud of particles is of order $\log N$, while most of the mass is concentrated at the left end of the cloud and spread over a distance of order 1. 
\item The relevant time scale both for macroscopic fluctuations of the cloud and for coalescence of ancestral lineages is $(\log N)^3$ \cite{Brunet2006,Brunet2006a,Berestycki2010,Maillard2013a}. The genealogy is asymptotically described by the Bolthausen--Sznitman coalescent \cite{Berestycki2010}.
\end{itemize}

By contrast, in our heavy-tailed setting, the following happens:
\begin{itemize}
\item The particles propagate either linearly or superlinearly (but still at most polynomially) in time. In both cases, the scaling factor due to the number of particles in the system grows roughly polynomially in $N$. 
%The particles propagate at a rate which is either superlinear (but polynomial) in time or linear, but with a speed which grows roughly polynomially in $N$. 
Without selection (i.e.\ for the classical BRW), the right-most particle would propagate exponentially fast in time \cite{Durrett1983} due to the exponential growth of the number of particles\footnote{For the stretched exponential case, see \cite{Gantert2000}.}.
% Similar behavior is observed in the heavy-tailed BRW without selection, which propagates at a rate polynomial in the number of particles at time $t$, i.e. exponentially in $t$ \cite{Durrett1983}. 
Note that a similar behavior is observed in reaction-diffusion equations with fractional Laplacian or other non-local operators \cite{Cabre2013}, which play the role of the FKPP equation in this context.
%\item The typical diameter is roughly polynomial in $N$, while 
\item The relevant time scale both for macroscopic fluctuations of the cloud and for coalescence of ancestral lineages is $\log N$. The genealogy is trivial\footnote{Here, \emph{trivial} means that the genealogy is given by the ``star-shaped coalescent''. This follows from the heuristic picture described above, although we haven't worked out a full proof of this fact.}. Moreover, the fluctuations come from large jumps of single particles, in contrast to the more complex mechanism leading to the fluctuations in the light-tailed setting \cite{Maillard2013a}.
%This is in contrast to the light-tailed setting, where fluctuations are due to the accumulation of comparatively large jumps spread over the particles in an ancestral lineage.
\end{itemize}

A natural question and possibility for future research is now to investigate what happens in between the two scenarios of light-tailed displacement distributions (satisfying an exponential moment assumption) and the polynomial tails considered in the present article.

\subsection{Organisation of the paper}
In Section~\ref{sec:stairs}, we first derive some basic properties of the stairs process and prove Theorem~\ref{th:stairs}, which gives its long-time behaviour. In Section~\ref{sec:brw_coupling}, we bound the $N$-BRW from below and from above by a discretised version of the stairs process. The main work here lies in the upper bound, which is contained in Proposition~\ref{prop:theta}. In the short Section~\ref{sec:dsp_coupling}, we couple the stairs process with its discretised version. Sections~\ref{sec:stairs}, \ref{sec:brw_coupling} and \ref{sec:dsp_coupling} can be read independently of one another. Finally, Section~\ref{sec:proofs} contains the proofs of Theorems~\ref{th:convergence} and \ref{th:main}, relying on the results obtained in the previous sections.

\section{Properties of the stairs process}
\label{sec:stairs}
Throughout this section, we assume that $\mu$ is a stairs measure as defined in the last section (i.e., $\mu$ is a non-zero measure on $(0,\infty)$ such that $\mu([a, +\infty[)<+\infty$ for every $a > 0$). We regard $\mu$ as an element in the space of $\sigma$-finite measures on $(0,\infty)$ endowed with the vague topology, i.e.\ the weak topology with respect to the space of continuous functions supported on a compact subset of $(0,\infty)$. Statements such as ``as $\mu$ varies'' always refer to this topology. We further denote by $\Rc(t)$ a realization of the $\mu$-stairs process.

We define a regeneration structure for $\Rc(t)$ as follows: Let $\tau_0=0$ and for $n\ge 1$, let $\tau_n = \inf\{t> \tau_{n-1}+1: \Rc(t) = \Rc(t-1)\}$. By the definition of the stairs process, the random times $(\tau_n)_{n\in\N}$ are regeneration times, i.e.\ the collection of pairs $(\Rc(\tau_n)-\Rc(\tau_{n-1}),\tau_n-\tau_{n-1})_{n\ge 1}$ are i.i.d. Furthermore, we have the following lemma:

\begin{proposition}
\label{prop:stairs_regen}
There exists a constant $C>0$ (depending on $\mu$), such that $\P(\tau_1 > t) \le C^{-1} e^{-Ct}$ for all $t>0$. This constant can be chosen to vary continuously with $\mu$.
\end{proposition}
\begin{proof}
By definition of $\mu$, there exists $m\in (0,\infty)$, such that $\mu([m,\infty)) < \infty$ and $\mu((2m,\infty)) > 0$. Let $M^1_n\ge M^2_n\ge\cdots$ be the $x$-coordinates of the atoms of the Poisson process from the definition of the $\mu$-stairs process in the time-interval $[n,n+1)$, arranged in decreasing order. We then have for each $n\ge0$,
\begin{equation}
\label{eq:reg}
\P(\tau_1 \le n+3\,|\,\F_n) \ge \P\left(M^1_n < m,\ M^1_{n+1} > 2m,\ M^2_{n+1} < m,\ M^1_{n+2} < m\right),
\end{equation}
because the event on the right-hand side ensures that there is a time $T\in[n+1,n+2)$ with $\Rc(T)-\Rc(T-) > m$ and with $\Rc(T+s) = \Rc(T)$ for all $s\in[0,1]$. This implies $\tau_1 \le T+1 \le n+3$. Now, the right-hand side of \eqref{eq:reg} is positive, independent of $n$ and continuous in $\mu$. The proposition is immediate.
\end{proof}

\begin{proposition}
\label{prop:stairs_integrable}
Suppose $\int_1^\infty x\mu(\dd x) <\infty$. Then the limit $\rho = \lim_{t\to\infty} \Rc(t)/t$ exists almost surely and in $L^1$ and satisfies $\rho = \E[\Rc(\tau_1)]/\E[\tau_1] > 0$.
\end{proposition}
\begin{remark}
\label{rem:rho_alpha}
Proposition~\ref{prop:stairs_integrable} implies in particular that for every $\alpha>1$, the limit $\rho_\alpha = \lim_{t\to\infty} \Rc^\alpha(t)/t$ exists almost surely and in $L^1$ and is a positive and finite constant.
\end{remark}
\begin{proof}[Proof of Proposition~\ref{prop:stairs_integrable}]
Let $\xi_t$ be as in the definition of the $\mu$-stairs process. For $n \geq 0$, define the process $\xi^n_t = \xi_{t}\Ind_{t>n}$ and let $(\Rc(n,t))_{t\in\R}$ be the $\mu$-stairs process defined from $\xi^n$ as in \eqref{eq:def_stairs}. One easily checks from the definition or by \eqref{eq:Rcrep} that $\Rc(0,m) \le \Rc(0,n)+\Rc(n,m)$ for every $n\le m$. Moreover, by the hypothesis on $\mu$, $\Rc(1)$ is positive and integrable, since $\P(\Rc(1) > x) = 1-\exp(-\mu((x,\infty)))$. Kingman's subadditive ergodic theorem \cite[Theorem 6.6.1]{Durrett1996}, whose remaining conditions are readily verified, now yields almost sure and $L^1$-convergence of  $\Rc(n)/n$ to a non-negative, finite constant $\rho$, and by the monotonicity of $\Rc(t)$, this convergence also holds for $\Rc(t)/t$.

Now note that $\E[\tau_1] < \infty$  by Proposition~\ref{prop:stairs_regen}, so that $\tau_n/n$ converges almost surely to $\E[\tau_1]$ by the law of large numbers. The almost sure convergence of $\Rc(t)/t$ established above then yields almost sure convergence of $\Rc(\tau_n)/n$ to $\E[\tau_1]\rho$. By the converse to the law of large numbers, this implies that $\E[\Rc(\tau_1)]$ is finite and that $\rho = \E[\Rc(\tau_1)]/\E[\tau_1]$. Moreover, $\E[\Rc(\tau_1)]$ is positive, because the measure $\mu$ is non-zero by definition and $\Rc(\tau_1) \ge \Rc(1)$. This shows \(\rho>0\). 

We remark that we could have proven Proposition~\ref{prop:stairs_integrable} without applying Kingman's subadditive ergodic theorem; using only the regeneration structure and an argument involving Fatou's lemma and the converse to the law of large numbers to get finiteness of $\E[\Rc(\tau_1)]$.
\end{proof}

\begin{proposition}
\label{prop:stairs_convergence}
Let $\mu_N$ be a sequence of stairs measures converging to $\mu$ and denote by $\Rc^N(t)$ and $(\tau^N_n)_{n\ge0}$ the corresponding stairs process and regeneration times. Then the following holds:
\begin{enumerate}
\item The sequence of processes $(\Rc^N(t))_{t\ge0}$ converges in law to $(\Rc(t))_{t\ge0}$ w.r.t.\ the $J_1$-topology (``Skorokhod's topology''), as $N\to\infty$.
\item The sequences of random variables $\tau^N_1$ and $\Rc^N(\tau^N_1)$ converge in law to $\tau_1$ and $\Rc(\tau_1)$, respectively, as $N\to\infty$.
\item The sequence $\tau^N_1$ is uniformly integrable in $N$.
\end{enumerate}
\end{proposition}
\begin{proof}
We can assume w.l.o.g.\ that $\mu$ has no atoms and has full support in $(0,\infty)$. Moreover, we can assume that $\mu_N((0,\infty)) \le \mu((0,\infty))$ for every $N$, otherwise we truncate $\mu_N$ near the origin. Let $\xi^N_t$ and $\xi_t$ be the processes used to construct $\Rc^N$ and $\Rc$. Define the functions $F,F_N:\R_+^*\to\R_+$ by $F(x) = \mu([x,\infty))$ and $F_N(x) = \mu_N([x,\infty))$ and let $F_N^{-1}(x) = \inf\{y\ge0:F_N(y)\le x\}$ be the generalised inverse of $F_N$. Defining the function $f_N := F_N^{-1}\circ F$, we can then couple the processes $\xi^N_t$ and $\xi_t$ by setting $\xi^N_t = f_N(\xi_t)$ (here we implicitly used the above-mentioned assumptions on $\mu$). 
%Note that for every $N\ge1$, the function $f_N$ is non-decreasing and varies slowly at $\infty$. 
%Moreover, by the uniform convergence theorem for slowly varying functions \cite[Theorem 1.2.1]{Bingham1987}, the sequence $(f_N)_{N\ge 1}$ converges to the identity function on $[0,\infty)$, uniformly on every compact interval $[0,r]$, $r>0$.
Note that for every $N\ge1$, the function $f_N$ is non-decreasing and the sequence $(f_N)_{N\ge 1}$ converges pointwise to the identity function on $[0,\infty)$.
Using \eqref{eq:Rcrep}, we can now show that almost surely, 
\begin{equation*}
%\label{e:cv-as-Nalpha}
\forall T\ge0: \lim_{N\to\infty}\sup_{0\le t\le T}|\Rc^N(t)-\Rc(t)| = 0,\ \lim_{N \to +\infty} \tau^N_1 = \tau_1,    \     \lim_{N \to +\infty}    \Rc^N(\tau^N_1) = \Rc(\tau_1).
\end{equation*}
This proves the first two claims. The third follows from Proposition~\ref{prop:stairs_regen}.
\end{proof}

For the next results, we suppose that $\mu_N$ is a sequence of stairs measures given by
\[
\mu_N([x,\infty)) = -\gamma_N  \log(1-h(c_Nx)^{-1}),
\] where $(\gamma_N)_{N\ge1}$ is a sequence such that $\gamma_N/(2N\log_2N)\to 1$ as $N\to\infty$. Here, $h$ is the function given in the introduction, in particular, $h(x)$ varies regularly at infinity with index $\alpha>0$. The sequence $\mu_N$ therefore converges to $\mu_\alpha$.

The following theorem gives the long-time behaviour of the $\mu_N$-stairs process. It will be important for the proof of Theorem~\ref{th:main}.

\begin{theorem}
\label{th:stairs}
We distinguish between the following cases:
\begin{itemize}
\item $\alpha > 1$: The limits $\rho_N = \lim_{t\to\infty} \Rc^N(t)/t$ and $\rho_\alpha = \lim_{t\to\infty} \Rc^\alpha(t)/t$ exist almost surely and in \(L^{1}\). Moreover, $\rho_N\to\rho_\alpha$ as $N\to\infty$.
\item $\alpha = 1$, $\E[X]<\infty$: The limit $\rho_N = \lim_{t\to\infty} \Rc^N(t)/t$ exists almost surely and in \(L^{1}\) and satisfies $\rho_N \sim \int_1^\infty h(c_N)/h(c_Nx)\,\dd x,$ as $N\to\infty$.
\item $\alpha = 1$, $\E[X] = \infty$: Set $b^N_t = \int_1^{h^{-1}(t)}h(c_N)/h(c_N x)\,\dd x$. Then, for every $N$,
\[
\Rc^N(t)/tb^N_t\to 1,\quad\text{in probability as $t\to\infty$.}
\]
% Then for every $\ep>0$ there exists $N_0$, such that for $N\ge N_0$,
%\[
%\limsup_{n\to\infty} \P\left(1-\ep < \frac{\log_2N}{c_N}\frac{\X_1(n)}{nb^N_n} \le \frac{\log_2N}{c_N}\frac{\X_N(n)}{nb^N_n} < 1+\ep\right) > 1-\ep.
%\]
\item $0 < \alpha < 1$:  Let  $W_\alpha$ be a random variable with Laplace transform given by \eqref{eq:Walpha}. Then,
\[
\lim_{N\to\infty} \limsup_{t\to\infty} d\left(\mathcal L\left(\frac{c_N}{h(c_N)^{1/\alpha}}\frac{\Rc^N(t)}{h^{-1}(t)}\right),\mathcal L(W_\alpha)\right) = 0
\]
\end{itemize}

\end{theorem}

A few remarks on Theorem~\ref{th:stairs}: In the case $\alpha > 1$, its proof is almost immediate from Propositions~\ref{prop:stairs_integrable} and \ref{prop:stairs_convergence} and Remark~\ref{rem:rho_alpha}. Indeed, since $\rho_N = \E[\Rc^N(\tau^N_1)]/\E[\tau^N_1]$ and $\rho_\alpha = \E[\Rc^\alpha(\tau^\alpha_1)]/\E[\tau^\alpha_1]$, it remains to show that the sequence of random variables $\Rc^N(\tau^N_1)$ is uniformly integrable in $N$, which can easily be done through fractional moment estimates using H\"older's inequality and Proposition~\ref{prop:stairs_regen}. It will also follow directly from Proposition~\ref{prop:Rc_tau} below. This proposition yields a precise estimate of the tail of $\Rc^N(\tau^N_1)$, which is the key to proving Theorem~\ref{th:stairs} in the more delicate case $\alpha \le 1$. Indeed, armed with Proposition~\ref{prop:Rc_tau}, the theorem directly follows from classic results on random walks applied to $(\Rc^N(\tau^N_n))_{n\ge 0}$.

We furthermore remark that the uniformity in $N$ in the statement of Proposition~\ref{prop:Rc_tau} is only needed in the case $\E[X]<\infty$. In the case $\E[X] = \infty$, estimates not uniform in $N$ would suffice.

\begin{proposition}
\label{prop:Rc_tau}
For every $\ep>0$, there exists $x_0=x_0(\ep)$, such that for all $N$,
\[
\sup_{x\ge x_0}\left|\frac{\P(\Rc^N(\tau^N_1) > x)}{\E[\tau^N_1]\mu_N([x,\infty))}-1\right| < \ep.
\]
\end{proposition}

Write $\Delta\Rc^N(t) = \Rc^N(t)-\Rc^N(t-)$ for the jump of $\Rc^N$ at time $t$. As part of the proof of Proposition~\ref{prop:Rc_tau}, we will show that if $\Rc^N(\tau^N_1)$ is large, then it is approximately equal to $\Delta\Rc^N(\tau^N_1-1)$. We therefore first prove the following lemma: 

\begin{lemma}
\label{lem:Delta_Rc}
For every $\ep>0$, there exists $x_0=x_0(\ep)$, such that for all $N$,
\[
\sup_{x\ge x_0}\left|\frac{\P(\Delta\Rc^N(\tau^N_1-1) > x)}{\E[\tau^N_1]\mu_N([x,\infty))}-1\right| < \ep.
\]
\end{lemma}

\begin{proof}
For better readability, write $\mu = \mu_N$, $\Rc = \Rc^N$, $\tau_1 = \tau^N_1$, $\Delta\Rc(t) = \Rc(t)-\Rc(t-)$, $\xi_t = \xi^N_t$.
We further set  $M_{t} = \max_{s\in(t-1,t]}\xi_s$ and $M_{t-} = \max_{s\in(t-1,t)}\xi_s$. 

Bounding $\P(\Delta\Rc(\tau_1-1) > x)$ from above is easy: Let $\Pc$ be the Poisson point process used to construct $\Rc$. We have
\[
\P(\Delta \Rc(\tau_1-1)>x) = \E\big[\sum_{(t,\xi_t)\in\Pc}\Ind_{\tau_1 = t+1,\,\Delta\Rc(t) > x}\Big] \le \E\big[\sum_{(t,\xi_t)\in\Pc}\Ind_{\tau_1 \ge t,\,\xi_t > x}\Big].
\]
Now, since $\tau_1$ is a stopping time for $\Rc$, the event $\{\tau_1 \ge t\} = \{\tau_1 < t\}^c$ is measurable w.r.t.\ the $\sigma$-field generated by $\Pc\cap \{[0,t)\times\R_+\}$ for every $t\ge0$. By the projection theorem for Poisson processes \cite[Theorem~VIII.T3]{Bremaud1981}, the previous equation now yields
\[
\P(\Delta \Rc(\tau_1-1)>x) \le (1-e^{-\mu([x,\infty))}) \int_0^\infty \P(\tau_1 \ge t)\,\dd t = (1-e^{-\mu([x,\infty))})\E[\tau_1],
\]
which proves the upper bound of  $\P(\Delta\Rc(\tau_1-1) > x)$.

As for the lower bound, we note that
for every $t\ge 0$, on the event $\{\Delta\Rc(t) > 0\}$, we have $\tau_1 \ge t+1$ iff $\tau_1 \ge t$. This gives for every $\ep>0$, $x>0$ and $t\ge 0$,
\begin{align*}
\{\tau_1 = t+1,\,\Delta\Rc(t) > x\} & \supset \{\tau_1 \ge t,\,M_{t+1}\le x,\,\Delta\Rc(t) > x\}\\
 &\supset \{\tau_1\ge t,M_{t+1}\le x,\,\,\xi_t>(1+\ep)x,\,M_{t-} \le \ep x\}.
\end{align*}
By the independence properties of the Poisson process and the above-mentioned projection theorem, this gives,
\begin{align*}
\P(\Delta \Rc(\tau_1-1)>x) &\ge \E\big[\sum_{(t,\xi_t)\in\Pc}\Ind_{\tau_1 \ge t,\,M_{t+1}\le x,\,\xi_t>(1+\ep)x,\,M_{t-} \le \ep x}\Big] \\
&\ge \P(M_1\le x)\E\big[\sum_{(t,\xi_t)\in\Pc}\Ind_{\tau_1 \ge t,\,\xi_t>(1+\ep)x}-\Ind_{\tau_1 \ge t-1,\,\xi_t>(1+\ep)x,\,M_{t-} > \ep x}\Big] \\
&=\P(M_1\le x)(1-e^{-\mu([(1+\ep)x,\infty))})(\E[\tau_1] - (1+\E[\tau_1])\P(M_1 > \ep x)).
\end{align*}
(in the second line, we used the fact that $\tau_1\ge t$ trivially implies $\tau_1\ge t-1$). Now, since $\mu_N$ converges to $\mu_\alpha$, the variables $M_1 = M^N_1$ converge in law as well. Together with the fact that $\E[\tau_1]\ge 1$ by definition, this gives for $x\ge x_0(\ep)$, for every $N$,
\[
\P(\Delta \Rc^N(\tau^N_1-1)>x) \ge (1-\ep)(1-e^{-\mu_N([(1+\ep)x,\infty))}).
\]
Now, by the regular variation of $h(x)$, we have for every $N$, for large $x$, $\mu_N([(1+\ep)x,\infty)) \ge (1+2\ep)^{-\alpha} \mu_N([x,\infty))$. Since $\ep$ was arbitrary, this proves the lemma.
\end{proof}

%\begin{lemma}
%For every $\ep>0$, we have as $x\to\infty$, $\P(\Delta \Rc(\tau_1-1) > x) \le \P(\Rc(\tau_1) > x) \le \P(\Delta \Rc(\tau_1-1) > (1-\ep)x)(1+o(1))$.
%\end{lemma}
\begin{proof}[Proof of Proposition~\ref{prop:Rc_tau}]
For better readability, we use the same notation as in the proof of Lemma~\ref{lem:Delta_Rc}. Note that we trivially have $\P(\Delta \Rc(\tau_1-1) > x) \le \P(\Rc(\tau_1) > x)$, such that the lower bound on $\P(\Rc(\tau_1) > x)$ directly follows from Lemma~\ref{lem:Delta_Rc}. 

For the upper bound, fix $\ep>0$. By the definition of the process $\Rc$, for every $x>0,$ the event $\{\Delta \Rc(\tau_1-1) \le (1-\ep)x,\,\Rc(\tau_1)>x\}$ implies the event $\{\Rc^-(\tau_1-1) > \ep x\}$, such that
\begin{equation}
\label{eq:121}
\P(\Rc(\tau_1) > x) \le \P(\Delta \Rc(\tau_1-1) > (1-\ep)x) + \P( \Rc^-(\tau_1-1) > \ep x).
\end{equation}
In order to bound the second term on the right-hand side,  let $C$ be large enough, such that with $L_x = \lceil C\log x\rceil$, we have $\P(\tau_1 > L_x) \le \mu([x,\infty))^2$ for large $x$ and every $N$ (this is possible by  Proposition~\ref{prop:stairs_regen} and Potter's bounds for regularly varying functions \cite[Theorem 1.5.6]{Bingham1987}). We then have for large
$x$,  with $\Rc^-(t) = \Rc(t-)$,
\begin{equation}
\label{eq:123}
\P( \Rc^-(\tau_1-1) >  x) \le \mu([x,\infty))^2 + \P( \Rc^-(\tau_1-1) > x,\,\tau_1\le L_x)
\end{equation}
We furthermore split the second term on the right-hand side of \eqref{eq:123} into two, according to whether $\max_{t<\tau_1-1}\xi_t > x/2$ or not. Now first note that if $\xi_t > x/2$ for some $t<\tau_1-1$, then necessarily $\xi_{t'} > x/4$ for some $t'<\tau_1-1$, otherwise $\tau_1 \le t+1$, which is a contradiction (this is the same reasoning as the one used in the proof of Proposition~\ref{prop:stairs_regen}). As a consequence, for some numerical constant $c$ we have for large $x$,
\begin{equation}
\label{eq:125}
\P(\max_{t<\tau_1-1}\xi_t > x/2,\,\tau_1\le L_x) \le c(L_x\mu([x,\infty)))^2.
\end{equation}
On the other hand, by \eqref{eq:Rcrep}, 
\[
\P( \Rc^-(\tau_1-1) > x,\,\max_{t<\tau_1-1}\xi_t \le x/2,\,\tau_1<L_x) \le \P(M_1+\cdots+M_{L_x} > x\,|\,\forall i:M_i \le x/2),
\]
and by the proof of Lemma~3 of \cite{Durrett1979} (see also the lemma stated on p.\,168 of \cite{Durrett1983}), the last quantity is less than $K_\alpha x^{-3\alpha/2}$ for a constant $K_\alpha$. Together with \eqref{eq:123} and \eqref{eq:125} and Potter's bounds \cite[Theorem 1.5.6]{Bingham1987}, this gives for large $x$, 
\[
\P( \Rc^-(\tau_1-1) >  x) \le \mu([x,\infty)) x^{-\alpha/3},
\]
and the regular variation of $h(x)$ then yields existence of $x_0$ (depending on $\ep$), such that $\P(\Rc^-(\tau_1-1) > \ep x) \le \ep \mu([x,\infty))$ for $x\ge x_0$. Together with \eqref{eq:121} and Lemma~\ref{lem:Delta_Rc}, as well as the regular variation of $h(x)$ and the fact that $\E[\tau_1]\ge1$ by definition, this finishes the proof.
\end{proof}  

\begin{proof}[Proof of Theorem~\ref{th:stairs}]
\uline{Case $\alpha > 1$:} Here, $\int_1^\infty x \mu_\alpha(\dd x) < \infty$ and for every $N$, $\int_1^\infty x\mu_N(\dd x) < \infty$ by Potter's bounds \cite[Theorem 1.5.6]{Bingham1987}. By Proposition~\ref{prop:stairs_integrable}, the limits  $\rho_N = \lim_{t\to\infty} \Rc^N(t)/t$ and $\rho_\alpha = \lim_{t\to\infty} \Rc^\alpha(t)/t$ exist almost surely and in \(L^{1}\) and satisfy $\rho_{N} = \E[\Rc^N(\tau^{N}_1)]/\E[\tau^{N}_1]$ and $\rho_\alpha = \E[\Rc^\alpha(\tau^\alpha_1)]/\E[\tau^\alpha_1]$. Furthermore, by Propositions~\ref{prop:stairs_regen} and~\ref{prop:stairs_convergence}, we have $\E[\tau^{N}_1]\to\E[\tau^\alpha_1]$ and $\Rc^N(\tau^{N}_1)$ converges in law to $\Rc^\alpha(\tau^{\alpha}_1)$, as $N\to\infty$. In order to show that $\rho_N\to\rho_\alpha$ as $N\to\infty$, it therefore remains to show that the sequence of random variables $\Rc^N(\tau^N_1)$ is uniformly integrable in $N$. But this follows from Proposition~\ref{prop:Rc_tau} and the fact that the restrictions of the measures $\mu_N$ to $[1,\infty)$ are uniformly integrable in $N$, again by  Potter's bounds \cite[Theorem 1.5.6]{Bingham1987}. This finishes the proof of the case $\alpha > 1$.

\uline{Case $\alpha = 1$, $\E[X]<\infty$:} As in the previous case, the limit  $\rho_N = \lim_{t\to\infty} \Rc^N(t)/t$ exists almost surely and in \(L^{1}\)  by Proposition~\ref{prop:stairs_integrable} and satisfies $\rho_{N} = \E[\Rc^N(\tau^{N}_1)]/\E[\tau^{N}_1]$. Furthermore, since the measures $\mu_N$ converge to $\mu_1(\dd x) = x^{-2}\,\dd x$ which satisfies $\int_1^\infty x\mu_1(\dd x) = \infty$, we have $\int_1^\infty x\mu_N(\dd x)\to\infty$, as $N\to\infty$. Proposition~\ref{prop:Rc_tau} now yields as $N\to\infty$ (note that here we need the fact that $x_0$ in the statement of  Proposition~\ref{prop:Rc_tau} is independent of $N$),
\[
\rho_N \sim \int_1^\infty \mu_N([x,\infty))\,\dd x \sim \int_1^\infty \frac{h(c_N)}{h(c_N x)}\,\dd x,
\]
which yields the theorem in the case $\alpha = 1$, $\E[X]<\infty$.

\uline{Case $\alpha = 1$, $\E[X]=\infty$:} For each $N\in\N$, define the random variable $S^N=\Rc^N(\tau^N_1)/\E[\tau^N_1]$ and set \(a^{N}_n = \inf\{x:\P(S^N>x)>n^{-1}\}.\)
We introduce the relation $a^N_n\asymp b^N_n$ between two positive double sequences meaning that $\lim_{N\to\infty}\limsup_{n\to\infty}|a^N_n/b^N_n-1|=0$. By Proposition~\ref{prop:Rc_tau}, 
\[
a^N_n \asymp (c_N\E[\tau^N_1])^{-1}h^{-1}(\E[\tau_1^N]h(c_N) n) \asymp c_N^{-1}h(c_N)h^{-1}(n),
\]
where the last relation follows from the fact that $h^{-1}$ is regularly varying with index 1 \cite[Theorem~1.5.12]{Bingham1987}.
Now define $\beta^N_n = \E[S^N\Ind_{S^N\le a^N_n}]$. By Proposition~\ref{prop:Rc_tau}, and the ``boundary case'' of Karamata's theorem \cite[Proposition~1.5.9a]{Bingham1987}, we have
\begin{align*}
\beta^N_n &= \int_0^{a_n^N} \P(S^N>x)\,\dd x - a_n^N\P(S^N>a_n^N)\\
&\sim \int_1^{a^N_n} \frac{h(c_N)}{h(c_N x)}\,\dd x \sim \int_1^{h^{-1}(n)} \frac{h(c_N)}{h(c_N x)}\,\dd x,\quad\text{as $n\to\infty$}.
\end{align*}
Furthermore, again by \cite[Proposition~1.5.9a]{Bingham1987}, we have  $n\beta^N_n/a^N_n\to\infty$ and  \(\beta^{N}_{n}\sim b^N_{cn}\) for every constant $c>0$, as $n\to\infty$. Now, if $(S^N_n)_{n\ge0}$ is a random walk with increments distributed according to $S^N$, then standard results on random walks (see e.g.\ \cite[Theorem~2.7.7]{Durrett1996}, note that $\P(S^N > x)$ is regularly varying for every $N$, by Proposition~\ref{prop:Rc_tau}) imply that $(S^N_n - n\beta^N_n)/a^N_n$ converges in law to a non-degenerate random variable, as $n\to\infty$. With the above, this implies that $S^N_n/n\beta^N_n\to 1$ in probability, as $n\to\infty$. Together with the fact that $\tau^N_n/n \to \E[\tau^N_1]$ almost surely as $n\to\infty$ and the monotonicity of $\Rc^N(t)$, this readily yields the theorem in the case $\alpha = 1$, $\E[X]=\infty$.

%{\color{red}TODO: By \cite[Theorem~8.8.1]{Bingham1987}, one also has $S^N_n/\gamma^N_n \to 1$ as $n\to\infty$, where $\gamma^N_n$ is defined by
%\[
%\frac 1 {\gamma^N_n} \int_1^{\gamma^N_n} \frac 1 {h(c_N x)}\,\dd x \sim \frac 1 n,\quad n\to\infty.
%\]
%It therefore should be true that $\gamma^N_n \sim n\beta^N_n$ but I was not able to show it. Do you find a proof? (this is a general question and does not depend on our current model: For a non-decreasing regularly varying function $h(x)$ of index 1, write $I(x) = \int_0^x 1/h(y)\,\dd y$. Is it true that $xI(h^{-1}(x))$ and $x/I(x)$ are asymptotically inverse? In other words, are $I(h^{-1}(x))$ and $1/I(x)$ de Bruijn conjugates?)}

\uline{Case $\alpha \in (0,1)$:} This case is similar to the previous one. Let $(S^N_n)_{n\ge0}$ be a random walk with increments distributed as $\Rc^N(\tau^N_1)$ and set \(a^{N}_n = \inf\{x:\P(S^N_1>x)>n^{-1}\}.\) Again by standard results on random walks (see e.g.\ \cite[Theorem
 XVII.5.3]{Feller1971} or \cite[Theorem~2.7.7]{Durrett1996}), since, by Proposition~\ref{prop:Rc_tau}, $\P(S^N > x)$ is regularly varying for every $N$, the sequence $S^N_n/a^N_n$ converges in law to $W_\alpha$, as $n\to\infty$, for every $N$. By Proposition~\ref{prop:Rc_tau},
\[
a^N_n \asymp \frac{h(c_N)^{1/\alpha}}{c_N} h^{-1}(\E[\tau^N_1]n).
\] 
Together with the fact that $\tau^N_n/n \to \E[\tau^N_1]$ almost surely as $n\to\infty$ and the monotonicity of $\Rc^N(t)$, this finishes the proof.
\end{proof}

We finish this section with an easy lemma which will be used in the proof of Theorem~\ref{th:convergence}.
\begin{lemma}
\label{lem:stoch_continuity}
The process $(\Rc(t))_{t\ge0}$ is stochastically continuous. Furthermore, the processes $(\Rc(t))_{t\ge0}$ and $(\Rc(t+1))_{t\ge0}$ almost surely do not have common jumps.
\end{lemma}
\begin{proof}
For every $t\ge0$ and $\delta > 0$, $\Rc(t+\delta)$ is by definition stochastically dominated by $\Rc(t)+\Rc'(\delta)$, for an independent copy $\Rc'$ of $\Rc$. This easily yields the first claim. For the second claim, we note that for every stopping time $T$ of the process $(\Rc(t))_{t\ge0}$, by the independence properties of the Poisson process, $T+1$ is almost surely not a jump time of $\Rc$. Fix $\ep>0$ and define $T_k$ to be the time of the $k$-th jump of size greater than $\ep$. Then $T_k\to\infty$ almost surely as $k\to\infty$ and almost surely, $T_k+1$ is not a jump time of $\Rc$ for every $k$. Letting $\ep\to0$ yields the lemma.
\end{proof}

\section{Coupling the BRW with a discretised stairs process}
\label{sec:brw_coupling}

Recall the definition of $X$, $(X_{n,i})_{n\ge0,i\in[2N]}$ and $\X(n) =  \{ \X_1(n) \leq \cdots \leq \X_N(n) \}$ from the introduction and define the rescaled variables $Y := c_N^{-1} X$, $Y_{n,i} := c_N^{-1} X_{n,i}$ and $\Y(n) := c_N^{-1} \X(n)$. For $n\ge1$, let $\F_n$ be the sigma-algebra generated by the variables $Y_{k,i}$ for $0\le k\le n-1$, $i\ge 1$ and set $\F_{-n}$ to the trivial $\sigma$-field for all $n\ge0$.  Note that $\X(n)$ and $\Y(n)$ are adapted to the filtration $(\F_n)_{n\ge0}$. For integers $N,\ell \geq 1$, the \emph{$(N,\ell)$-discretised stairs process} (DSP) is the process $(R^{N,\ell}(n))_{n\in\Z}$ defined inductively by $R^{N,\ell}(n)=0$ for all $n\le 0$ and
\begin{equation}
 \label{eq:dsp}
 R^{N,\ell}(n+1) = R^{N,\ell}(n)\vee \max_{i=1,\ldots,2N} (R^{N,\ell}(n-\ell)+Y_{n,i}).
\end{equation}

\subsection{Lower bound}
\label{sec:lower}

\begin{proposition}
\label{prop:lower}
 Let $\ell_N=\lceil \logtwo N \rceil$. Then $R^{N,\ell_N}(n-\ell_N)\le \Y_1(n)$ and \(R^{N,\ell_N}(n)\le \Y_N(n)\) for all $n\ge0$.
\end{proposition}
\begin{proof}
Writing $R(n) = R^{N,\ell_N}(n)$ for short, define the random time
\[
\tau = \inf\{n\ge 0:\Y_1(n) < R(n-\ell_N)\text{ or }\Y_N(n) < R(n)\}.
\]
%  This implies that
% \begin{equation}
%  \label{eq:1}
% \Y_1(n)\ge R(n-\ell_N)\quad\tand\quad \Y_N(n) \ge R(n)\quad \text{for all $n<T$.}
% \end{equation}
We shall prove by contradiction that  $\tau$ is infinite almost surely. Assume that $\tau$ is finite for some realisation of the variables $Y_{n,i}$. Set $\Y_i(-n) = 0$ for all $n>0$. By definition of $\tau$, one has $\Y_1(\tau-1) \ge R(\tau-1-\ell_N)$, so that by definition of $\Y(n)$ and $R(n)$, we then have $\Y_N(\tau)\ge R(\tau)$. We therefore must have $\Y_1(\tau) < R(\tau-\ell_N)$. But now, by the definition of $\tau$, we have $\Y_N(\tau-\ell_N) \ge R(\tau-\ell_N)$. Hence, at time $\tau$, there are no particles below $R(\tau-\ell_N)$, otherwise the particle at position $\Y_N(\tau-\ell_N)$ at time $\tau-\ell_N$ would have $2^{\ell_N}\ge N$ descendants at time $\tau$, all above $R(\tau-\ell_N)$, whence the total number of particles would be larger than $N$, which is a contradiction. Thus, $\Y_1(\tau) \ge R(\tau-\ell_N)$, but this contradicts the definition of~$\tau$.
\end{proof}

\subsection{Upper bound}
\label{sec:upper}

The main result in this section is Proposition~\ref{prop:theta}. One should not be fooled by its apparent simplicity, its proof is more intricate than it looks at first sight (and took us quite some time to come up with). Let  $\delta_N$ be a positive sequence which tends to zero as $N\to\infty$ but such that $\delta_N N^{\ep}\to\infty$ for all $\ep>0$. Set 
\(m_N = \logtwo N +\logtwo \delta_N\) and assume that $m_N\in\N$ and $m_N\ge 1$ for all $N$. Define the process $(\theta_n)_{n\ge0}$ by $\theta_0=0$ and
\[
 \theta_{n} = \max_{0\le k\le n}\{\Y_N(k)-R^{N,m_N}(k)\}.
\]

%Proposition~\ref{prop:upper} will now follow from the following result:

\begin{proposition}
\label{prop:theta}
For all $\ep>0$, for all large enough $N$ and for all $(1/2)\log_2 N \le m\le 2\log_2N$ and $x\ge m^{(1\vee \alpha^{-1})+\ep}/c_N$, for all $n\ge0$,
 \[\P(\theta_{n+1}-\theta_n > x\,|\,\F_{n-m}) \le Nm2^{m+1}/h(c_Nx)^2.\]
\end{proposition}

\begin{proof}
Let $n\in\N$ and $x>0$. Set $\Y_i^{(n)}(k) = \Y_i(k) - \theta_n$, $k=0,1,\ldots$ and note that $\Y_N^{(n)}(k) \le R^{N,m}(k)$ for all $0\le k\le n$.
From the definitions, one checks that the event $\theta_{n+1}-\theta_n > x$ implies that there exists $i\in[N]$ and $j\in\{0,1\}$, such that on the one hand $Y_{n+1,2i+j} > x$ and on the other hand $\Y_i^{(n)}(n) > x+R^{N,m}(n-m)\ge x+\Y_N^{(n)}(n-m)$, whence $\Y_i(n) > x+\Y_N(n-m)$ (this is best seen by a picture).

If we denote by $M_x(n)$ the number of particles above $x+\Y_N(n-m)$ at time $n$ (i.e.\ $M_x(n) = \#\{i\in[N]:\Y_i(n) > x+\Y_N(n-m)\}$), then a union bound gives,
\begin{equation}
 \P(\theta_{n+1}-\theta_n > x\,|\,\F_{n-m}) \le \E(M_x(n)\,|\,\F_{n-m})\P(Y>x) \le \E[M_x(m)]/h(c_N x).
\end{equation}
Let $(S_n)_{n\ge0}$ be a random walk with increments distributed according to $Y$. Bounding $M_x(m)$ by the number of particles above $x$ at time $m$ in $N$ branching random walks \emph{without} selection, we get
\begin{equation}
\E[M_x(m)] \le N 2^{m} \P(S_{m} > x).
\end{equation}

Now, for all $\ep>0$, for all large enough $N$ and for all $(1/2)\log_2 N \le m\le 2\log_2N$ and $x\ge m^{(1\vee \alpha^{-1})+\ep}/c_N$, we have $\P(S_{m}>x) \le 2m\P(Y> x)$ (see e.g.\ \cite[Theorem~3.3]{Cline1998}).
The statement follows.
\end{proof}

\begin{corollary}
\label{cor:upper}
Let $p\in [0,2\alpha)$. Then for every $0<\ep\le (2\alpha-p)/2$, there exists $N_\ep$, such that for $N>N_\ep$ and $n\ge0$, we have
\[
\E[(\theta_{n+1}-\theta_n)^p\,|\,\F_{n-m_N}] \le \left(1+\frac{4p}{2\alpha-p}\right) \left(\frac{\delta_N}{\log_2N}\right)^{\frac p {2\alpha + \ep}}.
\]
\end{corollary}
\begin{proof} Write $\E_n = \E[\cdot\,|\,\F_{n-m_N}]$ and set $\delta_N' = \delta_N/\log_2N$. By Proposition~\ref{prop:theta}, we have for every $x>\gamma m_N/c_N$ and every $n\ge0$:
\begin{equation*}
%\label{eq:theta}
\E_n[(\theta_{n+1}-\theta_n)^p]\le x+\frac {Nm_N2^{m_N+1}}{h(c_N)^2}\int_x^\infty \frac{h(c_N)^2}{h(c_Ny^{1/p})^2}\,\dd y.
\end{equation*}
By definition, we have $2^{m_N} = N\delta_N$ and $h(c_N) \sim 2N\log_2 N$ as $N\to\infty$. By Potter's bounds \cite[Theorem 1.5.6]{Bingham1987}, there now exists $x_\ep$, such that for $c_N\ge x_\ep$ and $x^{1/p}c_N\ge x_\ep$, we have %$h(c_N)^2/h(c_Ny^{1/p})^2 \le 2y^{-2\alpha/p}\max(y^{\ep/p},y^{-\ep/p}).$ ,
\begin{equation}
\label{eq:theta}
\E_n[(\theta_{n+1}-\theta_n)^p]\le x+\delta'_N I_x,\quad I_x=\int_x^\infty y^{-2\alpha/p}\max(y^{\ep/p},y^{-\ep/p})\,\dd y.
\end{equation}
By the hypothesis on $\ep$, we have for $x\le 1$: $I_x \le 4(p/(2\alpha-p))x^{1-(2\alpha+\ep)/p}$.
Setting now $x=x_N = (\delta'_N)^{p/(2\alpha+\ep)}$ in \eqref{eq:theta} yields the lemma (note that $x_N^{1/p}c_N \ge x_\ep$ and $x_N>m_N^{(1\vee \alpha^{-1})+\ep}/c_N$ for large $N$, by the hypothesis on $\delta_N$).
\end{proof}

\section{Coupling the discretised stairs process and the stairs process}
\label{sec:dsp_coupling}

Let $N,\ell\in\N$ and define the measure $\mu_{N,\ell}$ on $\R_+$ by \[\mu_{N,\ell}([x,\infty)) = -2N\ell\log(1-h(c_Nx)^{-1}).\] Let $\Rc^{\mu_{N,\ell}}(t)$ be the $\mu_{N,\ell}$-stairs process as defined in the introduction. Furthermore, let $R^{N,\ell}(t)$ be the $(N,\ell)$-DSP defined in \eqref{eq:dsp}. 
\begin{proposition}
\label{prop:dsp_stairs}
 We have $(\Rc^{\mu_{N,\ell}}(n/\ell))_{n\ge0} \stackrel{\text{st}}{\ge} (R^{N,\ell}(n))_{n\ge0}  \stackrel{\text{st}}{\ge}(\Rc^{\mu_{N,\ell}}(n/(\ell+1)))_{n\ge0},$ where $(X(n))_{n\ge 0} \stackrel{\text{st}}{\ge} (Y(n))_{n\ge 0}$ means that there exists a coupling, such that $X(n)\ge Y(n)$ for all $n\ge 0$.
\end{proposition}

\begin{proof}
Let $\xi_t$ be the function used to define the process $\Rc^{\mu_{N,\ell}}(t)$. By the definition of the measure $\mu_{N,\ell}$, the variables $\max\{\xi_t:t\in[0,1/\ell)\}$ and $\max\{Y_{1,i}:i\in[2N]\}$ have the same distribution, a fact which is also known as Serfling's coupling after \cite{Serfling1978} (see also \cite{Pfeifer1985}). 
We can therefore construct the process \(R^{N,\ell}(n)\) using $\xi_t$ by
\begin{equation}
\label{eq:Rn}
R^{N,\ell}(n+1) = R^{N,\ell}(n) \vee \max_{t\in[0,1/\ell)}(R^{N,\ell}(n-\ell) + \xi_{(n+1)/\ell-t}).
\end{equation}
In comparison, the processes $\Rc^{\mu_{N,\ell}}(t)$ and $\Rc^{\mu_{N,\ell}}(t)$ satisfy by definition
\begin{align}
\label{eq:RNt1}
\Rc^{\mu_{N,\ell}}(\tfrac {n+1} \ell) &= \Rc^{\mu_{N,\ell}}(\tfrac n \ell) \vee \max_{t\in[0,1/\ell)}(\Rc^{\mu_{N,\ell}}(\tfrac{n+1-\ell} \ell -t) + \xi_{(n+1)/\ell-t}).\\
\label{eq:RNt2}
\Rc^{\mu_{N,\ell}}(\tfrac{n+1}{\ell+1}) &= \Rc^{\mu_{N,\ell}}(\tfrac n {\ell+1}) \vee \max_{t\in[0,1/(\ell+1))}(\Rc^{\mu_{N,\ell}}(\tfrac {n-\ell}{\ell+1}-t) + \xi_{(n+1)/(\ell+1)-t}).
\end{align}
Equations~\eqref{eq:Rn} and \eqref{eq:RNt1} and the monotonicity of $\Rc^{\mu_{N,\ell}}(t)$ now directly yield the first inequality in the statement of the proposition. As for the second inequality, if we take \eqref{eq:RNt2} as the definition of the process $(\Rc^{\mu_{N,\ell}}(\tfrac n {\ell+1}))_{n\ge0}$, then exchanging $\xi_{(n+1)/(\ell+1)-t}$ by $\xi_{(n+1)/\ell)-t}$ does not change its law, and we obviously have $\max_{t\in[0,1/(\ell+1))}\xi_{(n+1)/\ell)-t} \le \max_{t\in[0,1/\ell)}\xi_{(n+1)/\ell)-t}$ for every $n$. Together with \eqref{eq:Rn} and the monotonicity of $\Rc^{\mu_{N,\ell}}(t)$, this yields the statement.
\end{proof}

\section{Proof of Theorems~\ref{th:convergence} and \ref{th:main}}
\label{sec:proofs}

The following lemma will be needed in the proof of Theorem~\ref{th:convergence}.

\begin{lemma}
\label{lem:upper_min}
We have for every $n_0\ge0$ and $\ep>0$,
\[
\P(\Y_1(n_0 + \lfloor (1-\ep)\log_2 N\rfloor) < \Y_N(n_0)+\ep\,|\,\F_{n_0}) \to 1,\quad\tas N\to\infty.
\]
\end{lemma}
\begin{proof}
Since we can bound the configuration of particles at time $n_0$ from above by moving all particles to the position of the maximum $\Y_N(n_0)$, it is clearly enough to show the lemma for $n_0=0$.
Let $\ep>0$ and set $\gamma_N=\ep/\log_2 N$.
Denote by $J_n$ the number of particles which jump by at least $\gamma_N$ between times $n$ and $n+1$. Then 
\(\E[J_n] = 2N\P(Y > \gamma_N) = 2Nh(c_N\gamma_N)\), such that $\E[J_n] \le N^{\ep/2}$ for large $N$, by Potter's bounds \cite[Theorem 1.5.6]{Bingham1987}). Now, if a particle is at a position strictly greater than $n\gamma_N$ at a time $n$, it must have an ancestor which has jumped by more than $\gamma_N$ between times $k-1$ and $k$ for some $k\le n$. This ancestor then has at most $2^{n-k}$ descendants at time $n$. Altogether, this gives for large $N$,
\[
\E[\#\{i: \Y_i(n) > n\gamma_N\}] \le \sum_{k=1}^n 2^{n-k}\E[J_k] \le 2^{n+1} N^{\ep/2},
\]
which now implies with $n_N=\lfloor (1-\ep)\log_2 N\rfloor$, for large $N$,
\[
\P(\Y_1(n_N) \ge \ep) \le \P(\#\{i: \Y_i(n_N) > n_N\gamma_N\} \ge N) \le 2N^{-\ep/2},
\]
by Markov's inequality. This yields the lemma.
\end{proof}

\begin{proof}[Proof of Theorem~\ref{th:convergence}]
Set  $(\Y_N'(t))_{t\ge0} = (\Y_N(\lfloor t\log_2 N\rfloor))_{t\ge0}$. We will first show that the finite-dimensional distributions of $(\Y_N'(t),\Y_1'(t))_{t\ge0}$  converge to those of  $(\Rc^\alpha(t),\Rc^\alpha(t-1))_{t\ge0}$. Recall the definitions of $\ell_N$ and $m_N$ from Sections~\ref{sec:lower} and \ref{sec:upper} and note that $\ell_N\sim m_N\sim \log_2 N$ as $N\to\infty$. For an upper bound, let $p=1$ if $\alpha \ge 1$ and $p\in(\alpha,\min(1,2\alpha))$ if $\alpha < 1$. Write $x_+ = \max(x,0)$ for $x\in\R$. We then have by Corollary~\ref{cor:upper}, for every $\ep>0$ and some $c>0,$  for large $N$,
\begin{equation}
\label{eq:551}
\forall n\ge0: \E[(\Y_N(n)-R^{N,m_N}(n))_+^p] \le \E[\theta_n^p] \le \sum_{k=1}^n\E[(\theta_k-\theta_{k-1})^p] \le c n \left(\frac{\delta_N}{\log_2 N}\right)^{p/(2\alpha+\ep)}.
\end{equation}
Fix $K>0$. If we choose $\delta_N = o((\log_2 N)^{1-(2\alpha+\ep)/p})$, then for all $n \le K\log_2 N$, the right-hand side in the last inequality tends to zero as $N\to\infty$. Together with Propositions~\ref{prop:stairs_convergence} and \ref{prop:dsp_stairs} as well as Lemmas~\ref{lem:stoch_continuity} and~\ref{lem:upper_min}, this shows that the finite-dimensional distributions of $(\Y_N'(t),\Y_1'(t))$ are tight in $N$ and are in the limit as $N\to\infty$ dominated by those of $(\Rc^\alpha(t),\Rc^\alpha(t-1))$. 

For a lower bound, note that by Propositions~\ref{prop:lower} and~\ref{prop:dsp_stairs}, we have for every $n\ge0$,
\begin{equation}
\label{eq:550}
(\Y_N(n),\Y_1(n)) \stackrel{\text{st}} \ge (R^{N,\ell_N}(n),R^{N,\ell_N}(n-\ell_N)) \stackrel{\text{st}} \ge  (\Rc^{\mu_{N,\ell_N}}(\tfrac {n} {\ell_N+1}),\Rc^{\mu_{N,\ell_N}}(\tfrac {n-\ell_N} {\ell_N+1})),
\end{equation}
with the coordinate-wise order on $\R^2$ (i.e.\ $(x,y)\le (v,w)$ iff $x\le v$ and $y\le w$).
Together with the first part of Proposition~\ref{prop:stairs_convergence} and Lemma~\ref{lem:stoch_continuity}, this proves that as $N\to\infty$, every limit point of the finite-dimensional distributions of $(\Y_N'(t),\Y_1'(t))_{t\ge0}$ dominates those of $(\Rc^\alpha(t),\Rc^\alpha(t-1))_{t\ge0}$. Together with the upper bound established above, this proves the convergence.

In order to prove tightness of $(\Y_N'(t))_{t\ge0}$ in Skorokhod's $J_1$-topology, we will use Aldous' criterion \cite[Theorem~1]{Aldous1978}: Let $T_N$ be a sequence of stopping times for $\Y_N'$. Suppose for simplicity that $T_N$ only takes on values which are multiples of $(\log_2N)^{-1}$. Let $\ep_N$ be a sequence of positive numbers converging to 0. We then have for every $x>0$,
\[
\P(\Y_N'(T_N+\ep_N)-\Y_N'(T_N) > x) \le \P(\Y_N'(\ep_N) > x),
\]
because we can bound the configuration of particles at time $T_N\log_2N$ from above by moving all particles to the position of the maximum. The right-hand side of the last inequality now converges to 0 by the convergence in finite-dimensional distributions established above together with the monotonicity of $\Y_N'(t)$ and  Lemma~\ref{lem:stoch_continuity}. By Aldous' criterion, this yields tightness in Skorokhod's $J_1$-topology.

As for the convergence of $(\Y_N'(t),\Y_1'(t))_{t\ge0}$ in the $SM_1$-topology, we note that by Skorokhod's representation theorem for stochastic processes \cite[{\S}3.1.2]{Skorokhod1956} and the convergence of the finite-dimensional distributions established above, we can transfer the processes $\Y_N'$, $\Y_1'$ and $\Rc^\alpha$ onto a common probability space, such that almost surely, $(\Y_N'(t),\Y_1'(t))\to (\Rc^\alpha(t),\Rc^\alpha(t-1))$ for every $t\in\Q_+$. The monotonicity of $\Y_N'$ and $\Y_1'$ then implies that almost surely, both $\Y_N'$ and $\Y_1'$ converge w.r.t.\ the $SM_1$-topology \cite[Corollary 12.5.1]{Whitt2002}. Convergence of the pair now follows from the second part of Lemma~\ref{lem:stoch_continuity}, by \cite[Theorem 12.6.1]{Whitt2002}.
\end{proof}

\begin{proof}[Proof of Theorem~\ref{th:main}]
We first cover the case $\E[X] < \infty$ (which includes the case $\alpha > 1$). The existence of the limit  $v_N = \lim_{N\to\infty} \X_N(n)/n = \lim_{N\to\infty} \X_1(n)/n$ is easily proven using subadditivity (see  \cite[Proposition~2]{Berard2010}) with the convergence holding almost surely and in $L^1$. The asymptotic for $v_N$ now easily follows from \eqref{eq:550} and \eqref{eq:551}, together with Theorem~\ref{th:stairs} and Proposition~\ref{prop:dsp_stairs}. Indeed, \eqref{eq:550} immediately gives a lower bound on $v_N$ and for the upper bound, we note that with  $\delta_N = o(\log N)^{1-2\alpha-\ep}$, the second term on the right-hand side of \ref{eq:551}, multiplied by $(\log_2N)/n$, vanishes in the limit as $N$ goes to infinity.

In the case $\E[X]=\infty$, set $\beta_n = nb^N_n$ if $\alpha=1$ and $\beta_n = h^{-1}(n)$ if $\alpha < 1$ and let $p=1$ if $\alpha = 1$ and $p\in(\alpha,\min(1,2\alpha))$ if $\alpha < 1$. Then $n/\beta_n^p\to 0$ as $n\to\infty$, by Potter's bounds \cite[Theorem 1.5.6]{Bingham1987}) and the fact that $h^{-1}(n)$ is regularly varying with index $1/\alpha$ \cite[Theorem~1.5.12]{Bingham1987}. Letting $\delta_N$ be any sequence satisfying the hypotheses of Corollary~\ref{cor:upper}, we get for every $N$ and every $\ep > 0$, for some constant $C_N$, by \eqref{eq:551},
\[
\P(\Y_N(n)-R^{N,m_N}(n) > \ep \beta_n) \le \ep^{-p}\beta_n^{-p}\E[(\Y_N(n)-R^{N,m_N}(n))_+^p] \le C_N\ep^{-p}n/\beta_n^{p}\to 0,
\]
as $n\to\infty$. This, together with Theorem~\ref{th:stairs} and Propositions~\ref{prop:lower} and \ref{prop:dsp_stairs}, implies the statement about $\X_N(n)$. The statement about $\X_1(n)$ follows from the fact that $\X_N(n-\lceil\log_2 N\rceil) \le \X_1(n) \le \X_N(n)$ for all $n$.
\end{proof}

\bibliography{n-brw-polynomial}
\end{document}